\documentclass[12pt,reqno]{amsart}

\usepackage[T1]{fontenc}
\pdfoutput=1
\usepackage{mathptmx,microtype,graphicx,dsfont}
\usepackage[hidelinks]{hyperref}
\usepackage{amsthm,amsmath,amssymb}
\usepackage{enumitem}
\usepackage[nameinlink,capitalise]{cleveref}
\usepackage[font=footnotesize,margin=2.6cm]{caption}
\usepackage{cite}

\usepackage[foot]{amsaddr}

\crefname{theorem}{Theorem}{Theorems}
\crefname{thm}{Theorem}{Theorems}
\crefname{mainthm}{Theorem}{Theorems}
\crefname{lemma}{Lemma}{Lemmas}
\crefname{lem}{Lemma}{Lemmas}
\crefname{remark}{Remark}{Remarks}
\crefname{claim}{Claim}{Claims}
\crefname{prop}{Proposition}{Propositions}
\crefname{proposition}{Proposition}{Propositions}
\crefname{defn}{Definition}{Definitions}
\crefname{corollary}{Corollary}{Corollaries}
\crefname{conjecture}{Conjecture}{Conjectures}
\crefname{question}{Question}{Questions}
\crefname{section}{Section}{Sections}
\crefname{figure}{Figure}{Figures}
\crefname{table}{Table}{Tables}

\theoremstyle{plain}

\newtheorem{thm}{Theorem}
\newtheorem*{thm*}{Theorem}
\newtheorem{lemma}[thm]{Lemma}

\newtheorem{prop}[thm]{Proposition}

\theoremstyle{definition}
\newtheorem{defn}[thm]{Definition}

\theoremstyle{remark}

\numberwithin{equation}{section}

\newcommand\cB{\mathcal{B}}
\newcommand\cG{\mathcal{G}}

\newcommand\cL{\mathcal{L}}
\newcommand\cM{\mathcal{M}}
\newcommand\cN{\mathcal{N}}
\newcommand{\cI}{{\mathcal I}}

\newcommand{\cT}{{\mathcal T}}
\newcommand{\cX}{{\mathcal X}}
\newcommand{\cW}{{\mathcal W}}

\newcommand\Z{\mathbb{Z}}

\renewcommand{\P}{{\mathbb P}}
\newcommand\E{{\mathbb E}}

\renewcommand{\le}{\leqslant}
\renewcommand{\ge}{\geqslant}

\author[M. Bassan]{Michal Bassan}
\address{University of Oxford, 
Department of Statistics
and Keble College}
\email{michal.bassan@keble.ox.ac.uk}

\author[S. Donderwinkel]{Serte Donderwinkel}
\address{University of Groningen, 
Bernoulli Institute for Mathematics, Computer Science and AI, 
and CogniGron (Groningen Cognitive Systems and Materials Center)}
\email{s.a.donderwinkel@rug.nl}

\author[B. Kolesnik]{Brett Kolesnik}
\address{University of Warwick, 
Department of Statistics}
\email{brett.kolesnik@warwick.ac.uk}


\begin{document}

\title[Graphical sequences and plane trees]
{Graphical sequences and plane trees}

\begin{abstract} 
Balister, the second author, 
Groenland, Johnston and Scott recently 
showed that there are
asymptotically $C4^n/n^{3/4}$ many
unordered sequences that occur as degree sequences of graphs.
Combining 
limit theory for infinitely divisible distributions
with 
a new 
bijective connection between a 
class of 
random walk trajectories
and a subset counting formula from additive 
number theory, we describe $C$ 
in terms of 
Walkup's number of 
rooted plane trees. 
The bijection is related to an instance of the 
L\'evy--Khintchine formula.
Our main result complements a result of 
Stanley, that ordered graphical sequences 
are related to quasi-forests. 
\end{abstract}

\maketitle

\section{Introduction}\label{S_intro}

In this work, we 
reveal a 
connection between trees embedded in the plane and
the degree sequences
of graphs. 
The proof involves an interplay 
between 
additive number theory,
random walks, renewal 
theory and infinite 
divisibility. 
In particular, we use  
the L\'evy--Khintchine formula, 
from the theory of L\'evy processes, 
to obtain a result 
in the area of combinatorial 
asymptotic enumeration. 
We expect this strategy to 
find more applications.

A sequence $d_1\le\cdots\le d_n$
is {\it graphical} if there is a graph on $n$ 
vertices with this degree sequence. Conditions under
which sequences are graphical are well-known; 
see Havel \cite{Hav55}, Hakimi \cite{Hak62} and 
Erd\H{o}s and Gallai \cite{EG60}. 
Recently, Balister, the second author, Groenland, Johnston and Scott
\cite{BDGJS22}
showed that 
the number $\cG_n$
of such sequences satisfies 
$n^{3/4}\cG_n/4^n\to C$. 
The constant $C$ is expressed as
a certain random walk probability $\rho$; 
see 
\eqref{E_BDGJS22} below. 

We show that 
$C$ can be expressed in terms of 
the number
$\cT_n$ 
of rooted unlabeled (cyclically distinct) plane trees with $n$ edges; see  \cref{F_T3} below.

\begin{thm}
\label{T_main}
As $n\to\infty$, we have that 
\[
\frac{n^{3/4}}{4^{n-1}}  \cG_n
\to 
\frac{e^\omega}{\Gamma(1/4)},
\]
where
\begin{equation}
\label{E_Walkup}
\omega=\sum_{k=1}^\infty
\frac{\cT_k}{k4^k}.
\end{equation} 
\end{thm}

As usual, 
$\Gamma(z)=\int_0^\infty t^{z-1}e^{-t}dt$ 
is the gamma function. 

Using P\'olya's enumeration theorem, 
Walkup \cite{Wal72} showed that 
\begin{equation}\label{E_Tn}
\cT_n
=
\frac{1}{n}
\sum_{d|n}{2d-1\choose d}\phi(n/d), 
\end{equation}
where $\phi(n)$ is Euler's totient function. Thus, we call $\omega$
{\it Walkup's constant}. 

By 
von Sterneck's subset counting formulas
(see, e.g., \cite{Bac02,Ram44}), 
$\cT_n$ is also the 
number of submultisets of
$\{0,1,\ldots,n-1\}$ of size $n$
that sum to $0$ mod $n$. 
This is key in our arguments, 
since, as we will see, it can be used to connect plane trees with 
lattice paths. 

\begin{figure}[h!]
\centering
\includegraphics[scale=1]{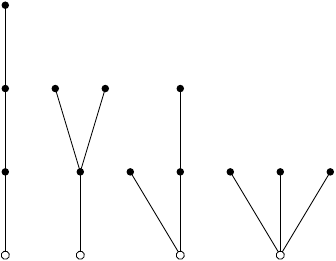}
\caption{The $\cT_3=4$ rooted unlabeled plane trees with 3 edges. 
}
\label{F_T3}
\end{figure}

\subsection{Degree sequences, via random walks}
\label{S_rho}

Erd\H{o}s and Moser 
observed a connection between 
random walks and  
degree sequences, 
in the context of graph tournaments, 
as discussed in Moon \cite{Moo68}.

A similar connection with random walks
is used in 
\cite{BDGJS22} to show that
\begin{equation}\label{E_BDGJS22}
\frac{n^{3/4}}{4^n}  \cG_n\to  
\frac{\Gamma(3/4)}{2^{5/2}\pi}
\frac{1}{\sqrt{1-\rho}}, 
\end{equation}
where $\rho$ is defined as follows. 

Let $Y_n$ be a lazy simple symmetric random walk,
started at $Y_0=0$, with increments 
$Y_{k+1}-Y_k$
equal to $\pm1$ with probability $1/4$, 
and $0$ with probability $1/2$. 
Let 
\[
\tau=\inf\{k\ge1:Y_k=0,\: A_k\le 0\}, 
\]
where
$A_k=\sum_{i=1}^k Y_i$ is the area
after $k$ steps. 
Then 
\begin{equation}\label{E_rho}
\rho=\P(A_\tau=0). 
\end{equation}

In proving \cref{T_main}, we 
obtain the following
description of $\rho$, in terms of 
Walkup's constant $\omega$. 

\begin{thm} 
\label{T_main2}
We have that 
$\rho=1-e^{-2\omega}$. 
\end{thm}

We note that \cref{T_main} follows by combining \cref{T_main2}  with 
\eqref{E_BDGJS22}, and noting that 
$\Gamma(3/4)=\pi\sqrt{2}/\Gamma(1/4)$. 

\subsection{Proof overview}\label{S_outline}

There is a natural representation of the 
lazy simple symmetric random walk $Y_n$, discussed above, 
in terms of a simple symmetric random walk
$X_n$ of twice the length. In this representation, the area process
$A_k=\sum_{i=1}^k Y_i=\frac{1}{2}\sum_{i=1}^k X_{2i}$. 
This relationship has a geometric 
interpretation; see the {\it diamond area}
discussed in 
\cref{S_GB} below.

If $X_{2n}=0$ and $A_k \ge0$ at all times $0\le k\le n$ for which $X_{2k}=0$, we call
$\{X_k,0\le k\le 2n\}$ a {\it graphical walk} of length $2n$. 
A \emph{graphical bridge} is a graphical walk whose area returns to $0$ 
at the end and a \emph{graphical meander} is a graphical walk 
whose area never returns to $0$. Times $k$ when $X_{2k}=A_k=0$ play the role of 
renewal times, 
since initially $X_0=A_0=0$. Hence a graphical walk
can be decomposed into a series of graphical bridges  
and one final graphical meander. That is,  
$\cW_n$ is a {\it delayed renewal sequence}, 
where the length of the final meander can be understood as the ``delay.'' 
The term $\Gamma(3/4)/2^{5/2}\pi$
in \eqref{E_BDGJS22} is related to the number 
$\cM_n$ of graphical meanders.

To describe 
the other term $1/\sqrt{1-\rho}$
in \eqref{E_BDGJS22}, we analyze the
renewal sequence (without a delay) $\cB_n$ of graphical bridges. 
We use the asymptotic transference theory, 
developed by Hawkes and Jenkins \cite{HJ78} 
(cf.\ Embrechts and Hawkes \cite{EH82}),
for infinitely divisible distributions. 
We call this method the 
{\it L\'evy--Khintchine method} (see 
\cref{S_ID} below), as it is 
related to the L\'evy--Khintchine transform, 
which associates each such distribution with
a corresponding L\'evy process, via another measure
called its L\'evy measure. 

We show that the L\'evy--Khintchine transform
$\cB_n^*$
of $\cB_n$ is related to $\cT_n$. In fact, $\cB_n^*=2\cT_n$.  
More specifically, we find that 
$p_n=e^{-2\omega}\cB^n/4^n$ is
an infinitely divisible probability distribution with L\'evy measure
$\nu_n=2\cT_n/n4^n$. The measure $\nu_n$ is regularly varying,
and it then follows by \cite{HJ78} that $p_n\sim \nu_n$. 

The proof that $\cB_n^*=2\cT_n$ involves 
a characterization (see \cref{lem:LKrenew} below) 
of the L\'evy--Khintchine transform of renewal sequences and some
delicate combinatorics, illustrated in 
\cref{F_bij,F_areas} below. Roughly speaking, 
by interleaving the steps of 
certain lattice paths, we relate
the set of cyclical shifts of graphical bridges
and the set of Walkup's plane trees, 
via von Sterneck's formula. 
Our arguments are geometric, 
using the diamond area  
mentioned above. 
See \cref{F_diamonds,F_walks5,F_phi}. 

Since $\cB_n$ is a renewal sequence, 
$\nu_n$ is also the 
harmonic renewal measure associated
with the 
sequence $\cB_n^{(1)}$ of {\it irreducible} graphical bridges. 
Using this, it follows 
that $e^{-2\omega}\nu_n/p_n$ is the 
harmonic moment $\E[1/\cI_n]$
of the number $\cI_n$ of irreducible graphical bridges
in a uniformly random graphical bridge of length $n$ 
and therefore $\E[1/\cI_n]\to e^{-2\omega}$ 
(since $p_n\sim \nu_n$). 

Moreover, it can be shown that $\cI_n$ converges to a 
negative binomial  with 
parameters $r=2$ and $p=1-\rho$. Intuitively, 
renewal times will only occur near the start and
end of the bridge. The numbers of such times on either side of the bridge 
are approximately independent and geometric
with $p=1-\rho$. Therefore, 
$\E[1/\cI_n]\to 1-\rho$, and so we find that 
$1-\rho=e^{-2\omega}$, yielding 
\cref{T_main2}. 
Hence, $1/\sqrt{1-\rho}=e^\omega$,
and so \cref{T_main} follows by 
\eqref{E_BDGJS22}.

\subsection{Ordered graphical sequences}
The polytope $D_n$ of degree sequences 
was introduced by 
Koren \cite{Kor73}, 
and studied by 
Peled and Srinivasan \cite{PS89}.
Stanley \cite[Corollary 3.4]{Sta91} found a formula
for the number 
of lattice points in $D_n\cap\Z^n$
with even sum, expressed in terms of quasi-forests. 
Such points are associated  with {\it ordered}
graphical sequences $d_1,\ldots,d_n$. 
On the other hand, \cref{T_main} 
relates the number of 
{\it unordered} graphical sequences 
$d_1\le\cdots\le d_n$
and Walkup's plane trees.

\subsection{Numerics}
Finally, let us note that the   
expression for $C$
in \cref{T_main}
can be used to approximate $\mathcal{G}_n$, for large $n$. 
Observing that the formula \eqref{E_Tn} for 
$\cT_n$ is dominated by the $d=n$ term, one can 
show that 
\[
\frac{n^{3/4}}{4^n}\cG_n
\to
C
\approx 
0.099094083237488745361449340935. 
\]
See 
the discussion of this work in \cite{BDGJS22}
for more details.

\subsection{Acknowledgments}
MB is supported by a Clarendon Fund Scholarship. 
SD acknowledges
the financial support of the CogniGron research center
and the Ubbo Emmius Funds (Univ.\ of Groningen).
BK was supported by a 
Florence Nightingale Bicentennial Fellowship (Oxford Statistics)
and a Senior Demyship (Magdalen College).

\section{The L\'evy--Khintchine method}
\label{S_ID}

A random variable $X$ is
{\it infinitely divisible} if, 
for all $n\ge1$, 
there are independent and identically distributed
$X_i$ such that $\sum_{i=1}^n X_i$ 
and $X$ are equal in distribution (see, e.g.,\cite{Fel68}). 

Suppose that a positive sequence 
$(1=a_0,a_1,\ldots)$ is summable, 
so that it is
proportional to a probability distribution $\pi_n$
on the non-negative integers $n\ge0$. 
As is well-known (see, e.g., \cite{HJ78})
such a $\pi_n$ is infinitely divisible 
if and only if, for some non-negative 
sequence $(a^*_1, a^*_2,\ldots)$, we have that 
\begin{equation}\label{E_AhatA}
\sum_{n=0}^\infty a_n x^n
=\exp\left(\sum_{k=1}^\infty\frac{a^*_k}{k}x^k\right). 
\end{equation}
In this case, we call $(a_0,a_1,\dots)$ an \emph{infinitely divisible sequence}.
By taking derivatives with respect to $x$ on both sides of \eqref{E_AhatA} 
and comparing coefficients, we see that (\ref{E_AhatA}) is equivalent to the recursion 
\begin{equation}\label{E_Arecursion}
na_n =\sum_{i=1}^n a^*_i a_{n-i},
\quad\quad n\ge1. 
\end{equation}

 Since, as discussed in \cite{EH82},  \eqref{E_AhatA}
is a special case of 
the L\'evy--Khintchine formula, 
we call $a_n^*$ the {\it L\'evy--Khintchine transform}
of $a_n$.
We note that, in fact, 
$\nu_n=a^*_n/n$ is the {\it L\'evy measure} 
(see, e.g., \cite{Ber96})
associated with the 
{\it L\'evy process} 
$\{\cX_t:t\ge0\}$, 
for which the law of the process $\cX_1$  
at time $t=1$ is described by 
the infinitely divisible probability measure 
\[
\pi_n
=a_n\exp\left(-\sum_{k=1}^\infty\frac{a^*_k}{k}\right),
\quad\quad n\ge0.
\]

In combinatorial settings, 
it is convenient to consider $a_n=A_n/\alpha^n$, 
where $A_n$ enumerates a class of objects of size $n$ 
that has exponential growth rate $\alpha$. 
If $a_n$ is infinitely divisible, then $A_n$ and 
$A^*_n=\alpha^n a_n^*$ also satisfy \eqref{E_AhatA} 
and \eqref{E_Arecursion}, and naturally we also call 
$A^*_n$ the L\'evy--Khintchine transform of $A_n$.

A positive sequence $\vartheta(n)$ is {\it regularly varying,}
with {\it index} $\gamma$, if
\[
\lim_{n\to\infty}\frac{\vartheta(\lfloor xn\rfloor)}{\vartheta(n)}
=x^\gamma,
\quad\quad \forall x>0. 
\]
One of our main tools for studying the asymptotics 
of sequences with a L\'evy--Khintchine transform is a result by 
Hawkes and Jenkins \cite{HJ78}, which shows that, if $a_n^*$ is regularly
varying with some index $\gamma<0$, then 
\begin{equation}\label{E_ID}
a_n
\sim 
\frac{a_n^*}{n}
\exp\left(\sum_{k=1}^\infty \frac{a^*_k}{k}\right).
\end{equation}
In other words, in terms of the underlying 
L\'evy process, the associated 
infinitely divisible probability measure $\pi_n$
and its corresponding L\'evy measure $\nu_n$ are asymptotically equivalent, $\pi_n\sim\nu_n$
as $n\to\infty$.

\subsection{Renewal sequences}

A subfamily of sequences $A_n$ that have a 
L\'evy-Khintchine transform is the class of \emph{renewal sequences}. 
Such a sequence arises when $A_n$ enumerates structures of length $n$ 
that can be decomposed into a sequence of irreducible parts. 
In this case, the generating function $A(x)=\sum_{n\ge 1} A_n x^n$ of $A_n$ satisfies 
\[A(x)=\frac{1}{1-A^{(1)}(x)}\]
where $A^{(1)}(x)=\sum_{n\ge 1} A^{(1)}_n x^n$ is the 
generating function of the irreducible structures (see, e.g., \cite{Fel68}). 

In our recent work \cite{BDK24}, we show that, 
when $A_n$ is a renewal sequence, 
$A^*_n$ takes a special form. 

\begin{lemma}[{\hspace{1sp}\cite[Lemma 2]{BDK24}}]
\label{lem:LKrenew}
Suppose that $A_n$ is a renewal sequence. 
Then: 
\begin{enumerate}
\item the L\'evy--Khintchine
transform $A_n^*$ is the number of pairs $(X,m)$, 
where $X$ is a structure of length $n$ and $0\le m <\ell$,
where $\ell=\ell(X)$ is the length of the first irreducible part of $X$, and 
\item we have that 
\begin{equation}\label{E_renewal}
\frac{A_n^*}{nA_n}
=\E\left[
\frac{1}{\cI_n}
\right],
\end{equation}
where $\cI_n$ is the number of irreducible 
parts in a uniformly random structure
of length $n$. 
\end{enumerate}
\end{lemma}

We note that, in terms of the underlying 
L\'evy process, 
the formula 
\eqref{E_renewal} says that,  
when $A_n$ is a renewal sequence with growth rate
$\alpha$, 
the L\'evy measure $\nu_n=A_n^*/n\alpha^n$
is the {\it harmonic renewal measure}
(see, e.g., \cite{GOT82})
associated with the measure 
$\mu_n=A_n^{(1)}/\alpha^n$.

With \cref{lem:LKrenew} in hand, 
one can use \eqref{E_ID} to deduce the asymptotic 
growth of a renewal sequence $A_n$ from the asymptotic growth of $A^*_n$, 
provided that there exists an $\alpha$ for which $A^*_n/\alpha^n$ is 
regularly varying with index $\gamma<0$. In our experience, 
the sequence $A^*_n$ tends to be more tractable than the 
renewal sequence $A_n$ itself, which lies at the core of our method.

\subsection{Delayed renewal sequences}

A related family of sequences $A_n$ is the class of \emph{delayed renewal sequences}. 
Such sequences arise when $A_n$ enumerates structures of length $n$ 
that can be decomposed into exactly one special part (the ``delay'') and 
a sequence of irreducible parts. In this case, the generating function $A$ of $A_n$ satisfies 
\[A(x)=\frac{D(x)}{1-A^{(1)}(x)},\]
where $D(x)$ is the generating function of the delay and $A^{(1)}(x)$ is the generating function of the irreducible structures. 

Although $A_n$ itself may not necessarily admit a L\'evy--Khintchine transform, 
we illustrate in this work that the asymptotic growth of $A_n$ can still 
be understood by studying the delay and its renewal structure separately. 
We believe that this method
will be useful in enumerating 
many other combinatorial sequences of interest.

\subsection{Application to $\cG_n$}

As discussed  in \cref{S_outline}, 
the sequence $\cG_n$ that enumerates graphical sequences is a delayed renewal sequence. 
The $1/\Gamma(1/4)$ factor in the asymptotics corresponds to the delay. 
To complete the proof of \cref{T_main}, 
it remains to consider the asymptotics
of the renewal sequence $\cB_n$ of
 graphical bridges of length $2n$, to which we will apply the L\'evy--Khintchine method, as 
described above. 
 
More specifically, we consider $a_n=\cB_n/4^n$,  
and its L\'evy--Khintchine transform $a^*_n$, and show that
the limit in \eqref{E_ID} equals $1-\rho$. 
Finally, we use \cref{lem:LKrenew} and a bijective argument to prove that 
$a^*_n=2\cT_n/4^n$, 
leading to the appearance of $\cT_n$ 
and $\omega$
in 
\cref{T_main,T_main2}. 

The two parts of \cref{lem:LKrenew} play 
an essential role. 
\cref{lem:LKrenew}(1) 
allows us to approach the
asymptotics of $\cB_n$ combinatorially, and \cref{lem:LKrenew}(2) 
opens the door to probability theory. 
These approaches meet halfway, so to speak, to give an explicit description of the asymptotics
of $\cG_n$.

\section{Graphical bridges}
\label{S_GB}

Suppose that $X=(X_0,\ldots,X_{2n})$
is a walk with $\pm1$ increments. 

\begin{defn}
We call 
\begin{equation}\label{E_sigma}
\sigma(X)=
\frac{1}{2}
\sum_{i=1}^n X_{2i}
\end{equation}
the {\it diamond area} of $X$. 
\end{defn}

In fact, $\sigma(X)$ is the usual area
of the lazy version $\Lambda(X)$ of $X$, 
defined as follows. 
Let $\Delta_i=X_i-X_{i-1}$, for $i\ge1$, be the increments of 
$X$. The increments $\Delta'_i$ of $\Lambda(X)$ are
the averages $\Delta'_i=(\Delta_{2i}+\Delta_{2i-1})/2$, 
for $i\ge1$. Note that, if $X$ 
has increments $\pm1$ with probability $1/2$, 
then $\Lambda(X)$ has increments 
$\pm1$ with probability $1/4$, 
and $0$ with probability $1/2$. 

The reason for the name ``diamond area'' is that 
$\sigma(X)$ is the signed area under the walk, 
that lies outside a 
rotated checkerboard
of diamonds, 
as in \cref{F_diamonds}. 

\begin{figure}[h!]
\centering
\includegraphics[scale=0.8]{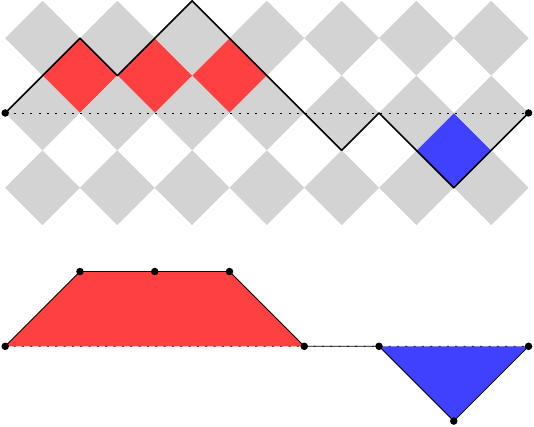}
\caption{A bridge $B$
with diamond area $\sigma(B)=2$, 
and its lazy version $\Lambda(B)$ below. 
}
\label{F_diamonds}
\end{figure}

For walks 
$X=(X_0,X_1,\ldots)$, 
we let $X^{(k)}=(X_0,X_1,\ldots,X_k)$.

\begin{defn}
A bridge $B=(B_0,\ldots,B_{2n})$ of length $2n$
with increments $\pm1$
is {\it graphical} 
if $\sigma(B)=0$ and 
$\sigma_{2k}=\sigma(B^{(2k)})\ge0$, for all $1\le k<n$. 
A graphical bridge $B$ is {\it irreducible} if 
$\sigma_{2k}>0$,  for all $1\le k<n$. 
We let $\cB_n$ be the number of graphical 
bridges of length $2n$.
See \cref{F_walks5}. 
\end{defn}

\begin{figure}[h!]
\centering
\includegraphics[scale=0.8]{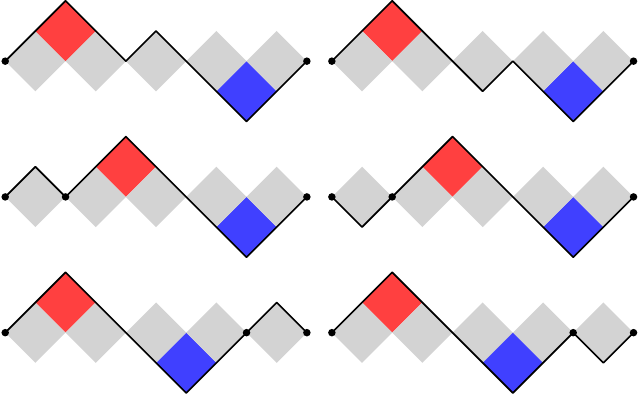}
\caption{There are $\cB_5=38$
graphical bridges of length $10$. 
Of these, 32 stay within 
the string of grey diamonds centered along the $x$-axis. 
The other 6 
are depicted above. 
The top two are irreducible. 
All others have two 
irreducible parts. Irreducible
parts are separated by solid dots. 
}
\label{F_walks5}
\end{figure}

\subsection{Relation to $\rho$}

Recall that $\rho$ is defined, in \cref{S_rho} above, in terms of 
a lazy random walk.  
We can rephrase $\rho$ in terms of a standard, 
simple symmetric random walk
$(X_n)$ 
with $\pm1$ increments, started from $X_0=0$, 
using its diamond area. 
Specifically, let $\sigma_{2k}=\sigma(X^{(2k)})$
denote the diamond 
area accumulated after $2k$ steps. 
Since, as discussed above, the diamond area
of $X$ is the usual area of its lazy version $\Lambda(X)$, 
it follows that 
\[
\rho=\P(\sigma_{2\kappa}=0),
\quad\quad\quad
\kappa=\inf\{k\ge1:X_{2k}=0,\: \sigma_{2k}\le 0\}. 
\]

 To see how $\rho$ relates to our problem,
we observe that a bridge $B=(B_0,\ldots,B_{2n})$ 
(with $\pm1$ increments) 
is graphical 
if and only if $\sigma(B)=0$ and 
$\sigma_{2k}=\sigma(B^{(2k)})\ge0$, for all times $2k$ 
for which $B_{2k}=0$, since $\sigma$ is monotone between 
such {\it diagnostic times}. 
Then, roughly speaking, $\rho$ is the probability 
that, the first time a random walk trajectory 
is in danger of not satisfying $\sigma_{2k}\ge0$, 
at such a diagnostic time $2k$, the condition is in fact 
satisfied with equality $\sigma_{2k}=0$. 

Consider now a random walk $X=(X_0,X_1,\ldots)$. 
If $X_{2k}=\sigma_{2k}=0$ at the first diagnostic time $2k<2n$, 
then 
$(X_0,\ldots,X_{2n})$ is a graphical bridge if and only if
$(X_{2k},\ldots,X_{2n})$ is a graphical bridge. 
Therefore, times at which $X_{2k}=\sigma_{2k}=0$ 
are of crucial importance, 
as they decompose a graphical bridge
into its irreducible parts, and correspond to {\it renewal times}
in the analysis. 

\begin{defn}
We let $\cI_n$
denote the number of irreducible parts
in a uniformly random 
graphical bridge
of length $2n$. 
\end{defn}

By (the proof of)
\cite[Proposition 21]{DK24a}  
it follows that 
\begin{equation}\label{E_NB}
\cI_n\overset{d}{\to}1+\cX,
\end{equation}
where $\cX$ is a negative binomial 
with parameters $r=2$ and $p=1-\rho$. 
The appearance of the negative binomial random variable can, 
intuitively, be explained as follows. With high probability, 
renewal times occur only very close to the 
start and end of the bridge. Moreover, 
the number of renewal times on either side
are approximately geometric
and independent.

\subsection{Calculating $\rho$}

We observed that 
a graphical bridge can be decomposed into 
a series of irreducible parts, and so $\cB_n$ is a renewal sequence. 
As discussed in \cref{S_ID}, this means that the generating function
$\cB(x)=\sum_n \cB_nx^n$ 
can be expressed as 
\[
\cB(x)=\frac{1}{1-\cB^{(1)}(x)},
\]
where $\cB^{(1)}(x)=\sum_n \cB_n^{(1)}x^n$ 
is the generating function for the number 
$\cB_n^{(1)}$ of {\it irreducible} graphical bridges
of length $2n$.

Therefore, by \cref{lem:LKrenew}(2), 
we have 
\begin{equation}\label{E_Ln1}
\frac{\cB_n^*}{n\cB_n}
=\E\left[\frac{1}{\cI_n}\right].
\end{equation}
Combining 
\eqref{E_NB} and \eqref{E_Ln1}, we find that 
\begin{equation}
\label{E_c1}
\frac{\cB^*_n}{n\cB_n}\to 1-\rho. 
\end{equation}
Hence, in proving \cref{T_main2}, 
the following is key. 

\begin{prop}
\label{P_LEt}
We have that
$\cB^*_n=2 \cT_n$. 
\end{prop}

We prove  
\cref{P_LEt}
 in the next two sections. 
 For now, let us note 
that our main result follows. 

\begin{proof}[Proof of \cref{T_main2}]
By \eqref{E_Tn}
and \cref{P_LEt}, 
\[
\frac{\cB^*_n}{4^n}
\sim\frac{1}{n4^n}{2n\choose n}
\sim \frac{1}{\sqrt{\pi}n^{3/2}}
\]
is regularly varying with
index $\gamma=-3/2$. 
Therefore, by \eqref{E_ID}, 
\begin{equation}
\label{E_c2}
\frac{\cB^*_n}{n\cB_n}
\to \exp(-2\omega).
\end{equation}
Combining \eqref{E_c1} and \eqref{E_c2},
we find that $\rho=1-\exp(-2\omega)$, as claimed. 
\end{proof}

\section{Combinatorial lemmas}
\label{S_comb}

In this section, we 
prove two combinatorial lemmas, 
showing that $2\cT_n$ can be described
in terms of areas $\alpha$ below  
lattice paths $L$ and diamond areas 
$\sigma$ under bridges $B$.

\subsection{Lattice paths}

Suppose that $L$ is an $\uparrow,\rightarrow$
lattice path from $(a_1,b_1)$ to $(a_2,b_2)$,
for some integers 
$a_1\le a_2$ and $b_1\le b_2$. 
We let $\alpha(L)$ denote the area of the region between 
$L$ and the lines $x=a_2$ and $y=b_1$.

\begin{defn}
We let  $\cN_n$ be the number of lattice paths $L$ 
from $(0,0)$ to $(n,n)$ such that 
$\alpha(L)\equiv0$ mod $n$. 
\end{defn}

To see the relationship
between $\cN_n$ and $\cT_n$, first note that
the area of a lattice path $L$
from $(0,0)$ to $(n,n)$ is 
\begin{equation}\label{E_alpha}
\alpha(L)=\sum_i u_i, 
\end{equation}
summing over the sequence $u_1\le\cdots\le u_n$, 
where 
$u_i$ is the number of $\uparrow$ steps before the $i$th 
$\rightarrow$ step of $L$. In other words, if we picture $L$ as
a bar graph, then the $u_i$ are the heights of the bars. 
Hence, $\cN_n$ is the number of submultisets
of $\{0,1,\ldots,n\}$ of size $n$ that sum to $0$ mod $n$. 
On the other hand, recall that $\cT_n$ is equal to the number of 
submultisets $\{0,1,\ldots,n-1\}$ of size $n$ that sum to $0$ mod $n$. 

Using these observation, we show the following. 

\begin{lemma}
\label{L_comb1}
We have that 
$\cN_n = 2\cT_n$. 
\end{lemma}

\begin{proof}
First, we claim that $\cT_n$ is the number of $\uparrow,\rightarrow$ lattice 
paths $L$ from $(0,0)$ to $(n,n-1)$ with $\alpha(L)\equiv0$
mod $n$. 
Indeed, submultisets $0\le u_1\le \cdots \le u_n\le n-1$ 
correspond to such 
$L$ with $\rightarrow$ steps at heights $u_i$, and 
so $\alpha(L)=\sum_i u_i\equiv0$
mod $n$. 

Let $\cL_n^\uparrow$ (resp.\ $\cL_n^\rightarrow$) be 
the number of lattice paths $L$ from $(0,0)$ to $(n,n)$ 
such that $\alpha(L)\equiv 0$ mod $n$, ending 
with an $\uparrow$ (resp.\ $\rightarrow$) step. 
Then $\mathcal{N}_n=\cL_n^\uparrow+\cL_n^\rightarrow$. 
To conclude,  
we show that 
$\cL_n^\uparrow=\cL_n^\rightarrow=\mathcal{T}_n$. 
Indeed, $\cL_n^\uparrow=\cL_n^\rightarrow$ 
follows by symmetry, reflecting over $x=y$, 
observing that divisibility by $n$ of the enclosed area is invariant under this reflection. 
Finally, $\cL_n^\uparrow=\mathcal{T}_n$, 
since if $L$ from $(0,0)$ to $(n,n)$ 
has last step $\uparrow$, then removing
this step yields a corresponding $L'$ from $(0,0)$ to $(n,n-1)$ 
with the same area as $L$. 
\end{proof}

\subsection{Bridges}

Next, we connect $\cN_n$ to bridges. 

\begin{defn}
We let $\cN_n'$ denote the number of bridges $B$
of length $2n$, with diamond area 
$\sigma(B)\equiv0$ mod $n$. 
\end{defn}

\begin{lemma}
\label{L_comb2}
We have that 
$\cN_n = \cN_n'$. 
\end{lemma}

\begin{proof}
To obtain a correspondence 
between bridges $B=(B_0,\ldots,B_{2n})$ of length $2n$ such that 
$\sigma(B)\equiv0$ mod $n$ and lattice paths $L$ 
from $(0,0)$ to $(n,n)$
with 
$a(L)\equiv0$ mod $n$, we proceed as follows.

First, let $\Delta_i = B_i-B_{i-1}$ be the $i$th increment of $B$. 
For $1\le k\le n$, put 
\[
B^{\mathrm{odd}}_k=\sum_{i=1}^k \Delta_{2i-1},
\quad\quad
B^{\mathrm{even}}_k=\sum_{i=1}^k \Delta_{2i},
\]
so that 
$B^{\mathrm{odd}}=(B^{\mathrm{odd}}_1,\ldots,B^{\mathrm{odd}}_n)$ 
and 
$B^{\mathrm{even}}=(B^{\mathrm{even}}_1,\ldots,B^{\mathrm{even}}_n)$ 
are the walks with 
the odd and even increments of $B$, respectively. 

Next, rotate $B^{\mathrm{odd}}$ 
counterclockwise by $\pi/4$
to obtain a lattice
path $L_1$ from $(0,0)$ to $(n-\ell,\ell)$ for some $\ell$ 
(since $B^{\mathrm{odd}}$ has $n$ steps).
Likewise, rotate $B^{\mathrm{even}}$ 
counterclockwise by $\pi/4$ to obtain a lattice walk $L_2$ 
from $(n-\ell,\ell)$ to $(n,n)$ (since $B$ is a bridge). 
Let $L$ be the concatenation   
of $L_1$ and $L_2$. 
This procedure is depicted in \cref{F_bij}.

Geometric considerations (see \cref{F_areas}) imply that
\[
\alpha(L)
=\alpha(L_1)+\alpha(L_2)+\ell^2.
\]
Furthermore, observe that, by definition,
\[
\sigma(B)
=\frac{1}{2}\sum_{i=1}^n \sum_{j=1}^i \left( \Delta_{2j-1} + \Delta_{2j}\right)
=\frac{1}{2}\sum_{k=1}^n (B^{\mathrm{odd}}_k+B^{\mathrm{even}}_k),
\]
which equals the signed area enclosed by 
$B^{\mathrm{odd}}$ and $-B^{\mathrm{even}}$, 
as depicted in \cref{F_areas}. 
Therefore, it can be seen that 
\[
\sigma(B)
=\alpha(L_1)-[\ell(n-\ell)-\alpha(L_2)]
=\alpha(L)-\ell n.
\]
Since $\alpha(L)\equiv \sigma(B)$ mod $n$, 
this completes the proof. 
\end{proof}

\begin{figure}[h!]
\centering
\includegraphics[scale=0.8]{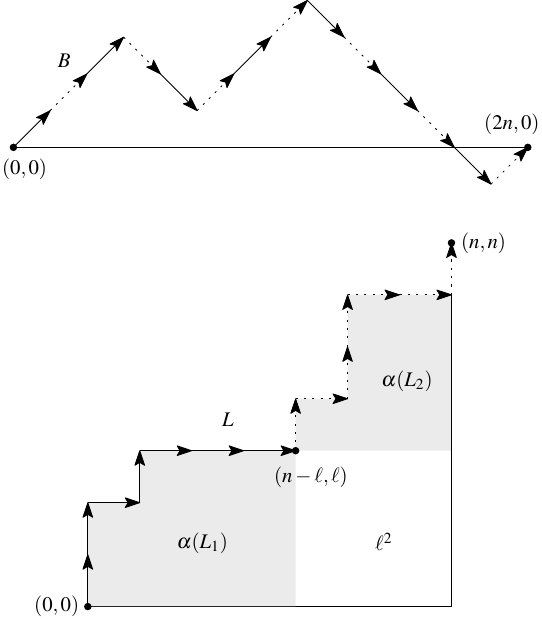}
\caption{The bijection in \cref{L_comb2}.}
\label{F_bij}
\end{figure}

\begin{figure}[h!]
\centering
\includegraphics[scale=0.8]{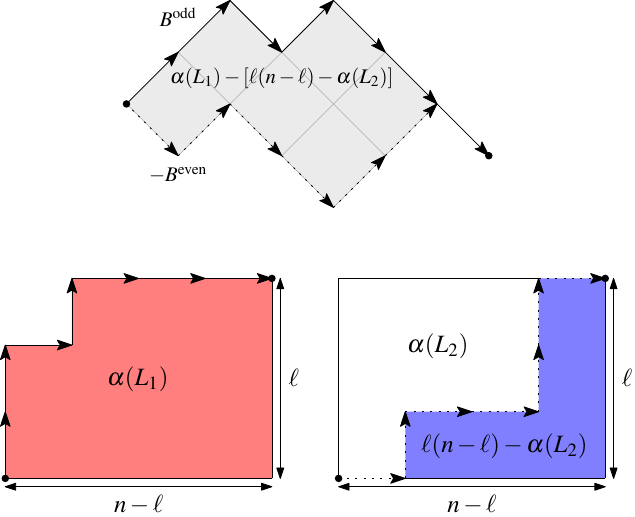}
\caption{Calculating areas in \cref{L_comb2}.}
\label{F_areas}
\end{figure}

\section{The L\'evy--Khintchine transform}\label{S_proof}

We are now in a position to
prove 
our key result \cref{P_LEt}, which identifies
$\cB_n^*= 2\cT_n$ as the L\'evy--Khintchine transform of 
$\cB_n$.

\begin{proof}[Proof of \cref{P_LEt}]
As discussed, 
$\cB_n$ is a renewal sequence, 
and so 
$\cB_n^*$ can be described in terms of 
\cref{lem:LKrenew}(1). 
Combining this with our combinatorial results 
\cref{L_comb1,L_comb2} above, 
it follows that, to show that $\cB_n^*=2\cT_n$, 
it suffices to find a bijection: 
\begin{itemize}
\item from the set
of ordered pairs 
$(B,i)$, where $B$ is a graphical bridge of length
$2n$, whose first irreducible part is of length $2\ell $, 
and $0\le i<\ell $
\item to the set of 
bridges $B'$
of length $2n$, with diamond area 
$\sigma(B')\equiv0$ mod $n$. 
\end{itemize}

We claim that such a bijection can be 
constructed 
as follows: 
For each such pair $(B,i)$, 
let $\phi(B,i)$ be the bridge $B'$ obtained
from $B$, via a cyclical shift to the right by $2i$.
Note that, geometrically, $B'$ can be obtained from 
$B$ by shifting the $x$-axis up or down by some multiple of $2$ 
and then starting the bridge from a particular intersection point of 
the new $x$-axis and the bridge $B$. Only the former operation 
affects the total diamond area. In fact, by \eqref{E_sigma}, 
it can be seen that such a shift changes the area by a multiple of $n$. 
Since $\sigma(B)=0$, it follows, 
again by 
\eqref{E_sigma}, that $\sigma(B')\equiv0$ mod $n$. 
See \cref{F_phi}. 

\begin{figure}[h!]
\centering
\includegraphics[scale=0.8]{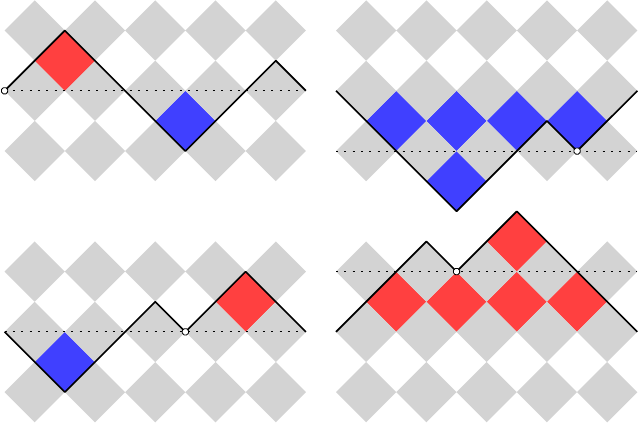}
\caption{A graphical bridge $B$ (top left) of length
10, with first irreducible part of length 8. 
The bridge $B$ is $\phi(B,0)$. Its 
other shifts
$\phi(B,i)$, for $1\le i<4$, are also depicted. 
All bridges have diamond area divisible by 5.
To find the inverse mapping $\phi^{-1}$, 
we shift the $x$-axis (see dotted line) by some factor of 2 
so that the diamond area is equal to 0, 
and then find the rightmost point (open dot) that starts a graphical sequence.  
Such a point exists by Raney's lemma. 
}
\label{F_phi}
\end{figure}

On the other hand, for any $B'$ with $\sigma(B')\equiv0$ mod $n$, 
the inverse $\phi^{-1}(B')$
is found as follows. Select the unique 
shift of the $x$-axis for which 
the diamond area becomes 0, and then choose
the rightmost starting point $2n-2i$ for which the 
resulting bridge $B$ is graphical; then $\phi^{-1}(B')=(B,i)$. Such
a point exists by Raney's lemma \cite{Ran60}, and by choosing $i$ 
minimal we ensure that $2i$ is smaller than the 
length of the first irreducible part of $B$.
\end{proof}

\makeatletter
\renewcommand\@biblabel[1]{#1.}
\makeatother

\providecommand{\bysame}{\leavevmode\hbox to3em{\hrulefill}\thinspace}
\providecommand{\MR}{\relax\ifhmode\unskip\space\fi MR }
\providecommand{\MRhref}[2]{%
  \href{http://www.ams.org/mathscinet-getitem?mr=#1}{#2}
}
\providecommand{\href}[2]{#2}


\begin{thebibliography}{10}

\bibitem{Bac02}
P.~G.~H. Bachmann, \emph{Niedere {Z}ahlentheorie}, vol.~2, Leipzig B.G.
  Teubner, 1902.

\bibitem{BDGJS22}
P.~Balister, S.~Donderwinkel, C.~Groenland, T.~Johnston, and A.~Scott,
  \emph{Counting graphic sequences via integrated random walks}, Trans. Amer. Math. Soc., to appear, preprint available at
  \href{https://arxiv.org/abs/2301.07022}{arXiv:2301.07022}.
 
\bibitem{BDK24}
M.~Bassan, S.~Donderwinkel, and B.~Kolesnik,
  \emph{Tournament score sequences,  
Erd\H{o}s--Ginzburg--Ziv numbers,
and the L\'evy--Khintchine method}, 
preprint available at
  \href{https://arxiv.org/abs/2407.01441}{arXiv:2407.01441}.
  
\bibitem{Ber96}
J.~Bertoin, \emph{L\'{e}vy processes}, Cambridge Tracts in Mathematics, vol.
  121, Cambridge University Press, Cambridge, 1996.

\bibitem{DK24a}
S.~Donderwinkel and B.~Kolesnik, \emph{Tournaments and random walks}, 
preprint available at
  \href{https://arxiv.org/abs/2403.12940}{arXiv:2403.12940}.  

\bibitem{EH82}
P.~Embrechts and J.~Hawkes, \emph{A limit theorem for the tails of discrete
  infinitely divisible laws with applications to fluctuation theory}, J.
  Austral. Math. Soc. Ser. A \textbf{32} (1982), no.~3, 412--422.

\bibitem{EG60}
P.~Erd\H{o}s and T.~Gallai, \emph{Gr\'afok el{\H{o}}\'irt foksz\'am\'u
  pontokkal}, Mat. Lapok \textbf{11} (1960), 264--274.

\bibitem{Fel68}
W.~Feller, \emph{An introduction to probability theory and its applications.
  {V}ol. {I}}, third ed., John Wiley \& Sons, Inc., New York-London-Sydney,
  1968.

\bibitem{GOT82}
P.~Greenwood, E.~Omey, and J.~L. Teugels, \emph{Harmonic renewal measures}, Z.
 Wahrsch. Verw. Gebiete \textbf{59} (1982), no.~3, 391--409.

\bibitem{Hak62}
S.~L. Hakimi, \emph{On realizability of a set of integers as degrees of the
  vertices of a linear graph. {I}}, J. Soc. Indust. Appl. Math. \textbf{10}
  (1962), 496--506.

\bibitem{Hav55}
V.~Havel, \emph{A remark on the existence of finite graphs}, Math. Bohemica
  \textbf{80} (1955), no.~4, 477--480.

\bibitem{HJ78}
J.~Hawkes and J.~D. Jenkins, \emph{Infinitely divisible sequences}, Scand.
  Actuar. J. (1978), no.~2, 65--76.

\bibitem{Kor73}
M.~Koren, \emph{Extreme degree sequences of simple graphs}, J. Combinatorial
  Theory Ser. B \textbf{15} (1973), 213--224.

\bibitem{Moo68}
J.~W. Moon, \emph{Topics on tournaments}, Holt, Rinehart and Winston, New
  York-Montreal, Que.-London, 1968.

\bibitem{PS89}
U.~N. Peled and M.~K. Srinivasan, \emph{The polytope of degree sequences},
  Linear Algebra Appl. \textbf{114/115} (1989), 349--377.

\bibitem{Ram44}
K.~G. Ramanathan, \emph{Some applications of {R}amanujan's trigonometrical sum
  {$C_m(n)$}}, Proc. Indian Acad. Sci., Sect. A. \textbf{20} (1944), 62--69.

\bibitem{Ran60}
G.~N. Raney, \emph{Functional composition patterns and power series reversion},
  Trans. Amer. Math. Soc. \textbf{94} (1960), 441--451.

\bibitem{Sta91}
R.~P. Stanley, \emph{A zonotope associated with graphical degree sequences},
  Applied geometry and discrete mathematics, DIMACS Ser. Discrete Math.
  Theoret. Comput. Sci., vol.~4, Amer. Math. Soc., Providence, RI, 1991,
  pp.~555--570.

\bibitem{Wal72}
D.~W. Walkup, \emph{The number of plane trees}, Mathematika \textbf{19} (1972),
  200--204.

\end{thebibliography}
\end{document}